\documentclass[12pt]{amsart}
\usepackage{amscd,amsmath,amsthm,amssymb}
\usepackage{amsfonts,amssymb,amscd,amsmath,enumerate,verbatim}
\usepackage[left]{lineno}
\usepackage{pstcol,pst-plot,pst-3d}

\newpsstyle{fatline}{linewidth=1.5pt}
\newpsstyle{fyp}{fillstyle=solid,fillcolor=verylight}
\definecolor{verylight}{gray}{0.97}
\definecolor{light}{gray}{0.9}
\definecolor{medium}{gray}{0.85}
\definecolor{dark}{gray}{0.6}

\def\NZQ{\mathbb}               

\def\ZZ{{\NZQ Z}}
\def\RR{{\NZQ R}}

\def\Lc{{\mathcal L}}
\def\G{{\mathcal G}}

\def\Ac{{\mathcal A}}

\def\Bc{{\mathcal B}}

\def\Pc{{\mathcal P}}
\def\Pc{{\mathcal P}}
\def\Pc{{\mathcal P}}

\def\xb{{\mathbf x}}
\def\yb{{\mathbf y}}

\def\eb{{\mathbf e}}

\def\vb{{\mathbf v}}
\def\opn#1#2{\def#1{\operatorname{#2}}} 
\opn\chara{char} \opn\length{\ell} \opn\pd{pd} \opn\rk{rk}
\opn\projdim{proj\,dim} \opn\injdim{inj\,dim} \opn\rank{rank}
\opn\depth{depth} \opn\grade{grade} \opn\height{height}
\opn\embdim{emb\,dim} \opn\codim{codim}
\opn\Cl{Cl}

\opn\Tr{Tr} \opn\bigrank{big\,rank}
\opn\superheight{superheight}\opn\lcm{lcm}
\opn\trdeg{tr\,deg}
	\opn\reg{reg} \opn\lreg{lreg} \opn\ini{in} \opn\lpd{lpd}
	\opn\size{size} \opn\sdepth{sdepth}
	\opn\link{link}\opn\fdepth{fdepth}\opn\lex{lex}
	\opn\tr{tr}
	\opn\type{type}
	\opn\gap{gap}
	\opn\arithdeg{arith-deg}
	\opn\revlex{revlex}
	\opn\div{div} \opn\Div{Div} \opn\cl{cl} \opn\Cl{Cl}
	\opn\Spec{Spec} \opn\Supp{Supp} \opn\supp{supp} \opn\Sing{Sing}
	\opn\Ass{Ass} \opn\Min{Min}\opn\Mon{Mon}
	\opn\Ann{Ann} \opn\Rad{Rad} \opn\Soc{Soc}
	\opn\Im{Im} \opn\Ker{Ker} \opn\Coker{Coker} \opn\Am{Am}
	\opn\Hom{Hom} \opn\Tor{Tor} \opn\Ext{Ext} \opn\End{End}
	\opn\Aut{Aut} \opn\id{id}
	
	\opn\nat{nat}
	\opn\pff{pf}
	\opn\Pf{Pf} \opn\GL{GL} \opn\SL{SL} \opn\mod{mod} \opn\ord{ord}
	\opn\Gin{Gin} \opn\Hilb{Hilb}\opn\sort{sort}
	\opn\PF{PF}\opn\Ap{Ap}
	\opn\mult{mult}
	\opn\bight{bight}
	\opn\div{div}
	\opn\Div{Div}
	\opn\aff{aff}
	\opn\relint{relint} \opn\st{st}
	\opn\lk{lk} \opn\cn{cn} \opn\core{core} \opn\vol{vol}  \opn\inp{inp} \opn\nilpot{nilpot}
	\opn\link{link} \opn\star{star}\opn\lex{lex}\opn\set{set}
	\opn\width{wd}
	\opn\Fr{F}
	\opn\QF{QF}
	\opn\G{G}
	\opn\type{type}\opn\res{res}
	\opn\conv{conv}
	\opn\Deg{Deg}
	\opn\Sym{Sym}
	\opn\Con{Con}
	\opn\gr{gr}
	
	\def\pot#1#2{#1[\kern-0.28ex[#2]\kern-0.28ex]}

	\opn\dirlim{\underrightarrow{\lim}}
	\opn\inivlim{\underleftarrow{\lim}}
	\let\to=\rightarrow
	
	\def\Implies{\ifmmode\Longrightarrow \else
		\unskip${}\Longrightarrow{}$\ignorespaces\fi}
	\def\implies{\ifmmode\Rightarrow \else
		\unskip${}\Rightarrow{}$\ignorespaces\fi}
	\def\iff{\ifmmode\Longleftrightarrow \else
		\unskip${}\Longleftrightarrow{}$\ignorespaces\fi}

	\let\:=\colon
	\newtheorem{Theorem}{Theorem}[section]
	\newtheorem{Lemma}[Theorem]{Lemma}

	\newtheorem{Example}[Theorem]{Example}

	\let\epsilon\varepsilon
	\let\kappa=\varkappa
	\def\qed{\ifhmode\textqed\fi
		\ifmmode\ifinner\quad\qedsymbol\else\dispqed\fi\fi}
	\def\textqed{\unskip\nobreak\penalty50
		\hskip2em\hbox{}\nobreak\hfil\qedsymbol
		\parfillskip=0pt \finalhyphendemerits=0}
	\def\dispqed{\rlap{\qquad\qedsymbol}}
	
	\opn\dis{dis}
	\def\pnt{{\raise0.5mm\hbox{\large\bf.}}}
	
	\opn\Lex{Lex}

\begin{document}

\title{Reflexive polytopes and discrete polymatroids}
\author {J\"urgen Herzog and Takayuki Hibi}
\address{J\"urgen Herzog, Fachbereich Mathematik, Universit\"at Duisburg-Essen, Campus Essen, 45117 Essen, Germany} 
\email{juergen.herzog@uni-essen.de}		
\address{Takayuki Hibi, Department of Pure and Applied Mathematics, Graduate School of Information Science and Technology, Osaka University, Suita, Osaka 565-0871, Japan}
\email{hibi@math.sci.osaka-u.ac.jp}
\dedicatory{ }
\thanks{The second author was supported by JSPS KAKENHI 19H00637.}
\begin{abstract}
A classification of discrete polymatroids whose independence polytopes are reflexive will be presented. 
\end{abstract}	
\maketitle
\thispagestyle{empty}

\section*{Introduction}
The discrete polymatroid is introduced in \cite{HH}.   In the present paper, as a supplement to \cite{HH}, a classification of discrete polymatroids whose independence polytopes are reflexive will be presented.  We refer the reader to \cite{HH} and \cite{HHgtm260} for fundamental materials on discrete polymatroids. 

\section{Reflexive polytopes}
	 
A convex polytope $\Pc \subset \RR^d$ of dimension $d$ is called a {\em lattice polytope} if each of its vertices belongs to $\ZZ^d$.  A {\em reflexive polytope} is a lattice polytope $\Pc \subset \RR^d$ of dimension $d$ for which the origin of $\RR^d$ belongs to the interior of $\Pc$ and the dual polytope $\Pc^\vee = \{\xb \in \RR^d : \langle \xb, \yb \rangle \leq 1, \forall \yb \in \Pc\}$ of $\Pc$ is a lattice polytope, where $\langle \xb, \yb \rangle$ stands for the canonical inner product of $\RR^d$.  A lattice polytope which can be a reflexive polytope by parallel shift is also called reflexive.

Let $\eb_1, \ldots, \eb_d$ denote the canonical basis vectors of $\RR^d$.  Let $P \subset \ZZ^d_{+}$ be a {\em discrete polymatroid} \cite[Definition 2.1]{HH} on the ground set $[d]$.  In what follows one assumes that each $\eb_i$ belongs to $P$.  Let $\Pc = \Pc_P \subset \RR^d$ denote the lattice polytope which is the convex hull of $P$ in $\RR^d$.  We call $\Pc$ the {\em independence polytope} of $P$.  One has $\dim \Pc = d$.  Let $\rho = \rho_P$ denote the {\em ground set rank function} \cite[pp.~243]{HH} of $\Pc$.  It follows from \cite[Theorem 7.3]{HH} that    

\begin{Lemma}
\label{aaaaa}
The independence polytope $\Pc$ is reflexive if and only if, for each subset $X \subset [d]$ which is $\rho$-closed and $\rho$-inseparable \cite[pp.~257--258]{HH}, one has $\rho(X) = |X| + 1$. 
\end{Lemma}

A {\em sublattice} of $2^{[d]}$ is a collection $\Lc$ of subsets of $[d]$ with $\emptyset \in \Lc$ and $[d] \in \Lc$ such that, for all $A$ and $B$ belonging to $\Lc$, one has $A \cap B \in \Lc$ and $A \cup B \in \Lc$. 

\begin{Theorem}
\label{classification}
(a) Let $P$ be a discrete polymatroid on the ground set $[d]$ and $\rho = \rho_P$ the ground set rank function of $\Pc$.  Let $\Ac$ be the set of $\rho$-closed and $\rho$-inseparable subsets of $\Pc$.  If $\Pc$ is reflexive, then $\Ac \cup \{\emptyset\}$ is a sublattice of $2^{[d]}$.

(b) Conversely, given a sublattice $\Lc$ of $2^{[d]}$, there exists a unique discrete polymatroid $P$ on the ground set $[d]$ for which $\Lc$ is the set of $\rho$-closed and $\rho$-inseparable subsets of $\Pc$ and $\Pc$ is reflexive.  
\end{Theorem}

\begin{proof}
(a)  If the independence polytope $\Pc$ of $P$ is reflexive, then Lemma \ref{aaaaa} says that $\rho(A) = |A| + 1$ for each $A \in \Ac$.  It follows from \cite[Proposition 7.2]{HH} that $\Pc$ consists of those $(x_1, \ldots, x_d) \in \RR^d$ for which 
\[
x_i \geq 0, \, \, \, \, \, i = 1, 2, \ldots, d,
\] 
and 
\begin{eqnarray}
\label{facet}
\sum_{i \in A} x_i \leq |A| + 1, \, \, \, \, \, A \in \Ac.
\end{eqnarray}
Since each $\eb_i \in \Pc$ and $\Pc$ is compact, it follows that
\[
\bigcup_{A \in \Ac} A = [d].
\]
Furthermore, if $X \not\in \Ac$, then 
$\rho(X) > |X| + 1$.  

In fact, if $|X|$ = 1 and $X = \{i\}$, then $3 \eb_i \in \Pc$ and $\rho(X) > 2$.  
In general, if $|X| = q \geq 2$ and $X = \{i_1, \ldots, i_q\}$ with $i_1 < \cdots < i_q$, then, 
one has 
\begin{eqnarray}
\label{point}
\vb = \frac{q}{\,q - 1\,}\sum_{j=1}^q\eb_{i_j} \in \Pc 
\end{eqnarray}  
and 
\[
\rho(X) \geq  |\vb| = \frac{q^2}{\,q - 1\,} > q + 1.
\]

To see why (\ref{point}) holds, one shows that $\vb$ satisfies each of the inequalities (\ref{facet}).  Let $A \in \Ac$ with $X \subsetneq A$, then 
\[
\frac{q^2}{\,q - 1\,} \leq q+2 = |X| + 2 \leq	 |A| + 1.  
\]  
Let $A \in \Ac$ with $|X \cap A| = k < q$, then  
\[
k \frac{q}{\,q - 1\,} \leq k + 1 \leq |A| + 1.
\]

One claims that $\Ac \cup \{\emptyset\}$ is a sublattice of $2^{[d]}$.  Let $A, B \in \Ac$ and suppose that either $A \cup B \not\in \Ac \cup \{\emptyset\}$ or $A \cap B \not\in \Ac \cup \{\emptyset\}$.  Then
\[
\rho(A) + \rho(B) = |A| + |B| + 2 = |A \cup B| + |A \cap B| + 2 < \rho(A \cup B) + \rho(A \cup B),
\]
which contradict the fact that $\rho$ is submodular.  Furthermore, since $\bigcup_{A \in \Ac} A = [d]$, one has $[d] \in \Ac$, as desired. 

\medskip

(b) By virtue of \cite[Theorem 9.1]{HH} one introduces the nondecreasing submodular function $\rho: 2^{[n]} \to \ZZ_+$ by setting   
\[
\rho(X) = \min\{|A| + 1 : X \subseteq A, A \in \Lc\}, \, \, \, \, \, \emptyset \neq X \subset [d]
\] 
together with $\rho(\emptyset) = 0$.  Let $P$ be the discrete polymatroid on the ground set $[n]$ and $\rho$ the ground set rank function of $\Pc$.  Then $\Lc$ is the set of $\rho$-closed and $\rho$-inseparable subsets of $[d]$.  Furthermore, Lemma \ref{aaaaa} guarantees that the independence polytope $\Pc$ of $P$ is reflexive.  On the other hand, suppose that $P'$ is a discrete polymatroid on the ground set $[d]$ and $\rho'$ its ground set rank function of the independence polytope $\Pc'$ of $P'$ for which $\Lc$ is the set of $\rho'$-closed and $\rho'$-inseparable subsets of $P'$ and for which $\Pc'$ is reflexive.  Then by using Lemma \ref{aaaaa} again one has $\rho'(A) = |A| + 1$ for each $A \in \Lc$.  Hence $\Pc = \Pc'$ (\cite[Proposition 7.2]{HH}).  Thus $P = P'$ (\cite[Theorem 3.4]{HH}), as desired.
\end{proof}

\section{Examples}

\begin{Example}
\label{exampleA}
{\em
Let $\Pc \subset \RR^3$ be the convex polytope whose facets are each $x_i = 0$ together with
\[
x_1 + x_2 = 3, \, \, \, x_2 + x_3 = 3, \, \, \, x_1 + x_2 + x_3 = 4.
\]
It can be checked that $\Pc$ is reflexive.  However, $\Pc$ cannot be the independence polytope of a discrete polymatroid on the ground set $[3]$.  In fact, if $\Pc$ is the independence polytope of a discrete polymatroid $P$ on the ground set $[3]$, then both $u = (0,3,0)$ and $v = (1,2,1)$ belong to the set of bases \cite[p.~245]{HH} of $P$.  One has $|u| < |v|$, which contradicts \cite[Theorem 2.3]{HH}.
}
\end{Example}

\begin{Example}
\label{exampleB}
{\em
Let $\Lc$ be a chain of length $d$ of $2^{[d]}$, say,
\[
\Lc = \{ \emptyset, \{d\}, \{d-1,d\}, \ldots, \{1, \ldots, d\}\}.
\]
Let $P$ denote the discrete polymatroid constructed in Theorem \ref{classification} (b).  Let
\[
\Bc = \{[d], [d], [d-1], \ldots, [2],[1]\}.
\]
Let $P'$ denote the transversal polymatroid \cite[p.~267]{HH} presented by $\Bc$.  If $X \subset [d]$ and $i = \min(X)$, then it follows from the proof of Theorem \ref{classification} (b) that
\[
\rho_{P}(X) = |\{i, i+1, \ldots, d\}| + 1 = (d - (i-1)) + 1 = d + 2 - i 
\]
On the other hand, by the definition of the ground set rank function of a transversal polymatroid, one has 
\[
\rho_{P'}(X) = (d + 1) - (i - 1) = d + 2 - i.
\] 
Hence $\rho_{P} = \rho_{P'}$.  Thus $P = P'$.
}
\end{Example}

It would be of interest for which sublattice $\Lc$ of $2^{[d]}$ the discrete polymatroid constructed in Theorem \ref{classification} (b) can be a transversal polymatroid.

\end{document}